\documentclass{article}

\usepackage{amsmath}
\usepackage{amssymb}
\usepackage{amsthm}
\newtheorem{lem}{Lemma}
\newtheorem{thm}{Theorem}
\newtheorem{dfn}{Definition}
\newtheorem{cor}{Corollary}

\title{Congruence classes of large configurations in vector spaces over finite fields}
\author{Alex McDonald}

\begin{document}

\maketitle

\begin{abstract}
In \cite{GroupActions}, Bennett, Hart, Iosevich, Pakianathan, and Rudnev found an exponent $s<d$ such that any set $E\subset \mathbb{F}_q^d$ with $|E|\gtrsim q^s$ determines $\gtrsim q^{\binom{k+1}{2}}$ congruence classes of $(k+1)$-point configurations for $k\leq d$.  Because congruence classes can be identified with tuples of distances between distinct points when $k\leq d$, and because there are $\binom{k+1}{2}$ such pairs, this means any such $E$ determines a positive proportion of all congruence classes.  In the $k>d$ case, fixing all pairs of distnaces leads to an overdetermined system, so $q^{\binom{k+1}{2}}$ is no longer the correct number of congruence classes.  We determine the correct number, and prove that $|E|\gtrsim q^s$ still determines a positive proportion of all congruence classes, for the same $s$ as in the $k\leq d$ case.
\end{abstract}

\section{Introduction}

The Erdos-Falconer distance problem in $\mathbb{F}_q^d$ asks if any sufficiently large set $E\subset\mathbb{F}_q^d$ must determine all possible distances, or more generally a positive proportion of all distances.  More precisely, given a set $E\subset\mathbb{F}_q^d$, for $d\geq 2$, define the distance set of $E$ by

\[
\Delta(E):=\{(x_1-y_1)^2+\cdots +(x_d-y_d)^2:x,y\in E\}.
\]

Clearly $\Delta(E)\subset\mathbb{F}_q$.  The Erdos-Falconer distance problem asks if there is an $s<d$ such that $|E|\gtrsim q^s$ implies $\Delta(E)=\mathbb{F}_q$, or more generally $\Delta(E)\gtrsim q$. \footnote{Throughout, we use the notation $X\gtrsim Y$ to mean there is a constant $c$ such that $X>cY$.  Likewise, $X\approx Y$ means there are constants $c_1$ and $c_2$ such that $c_1Y\leq X\leq c_2 Y$.}  \\

This problem was studied by Iosevich and Rudnev in \cite{IR}.  In that paper, the authors prove that $\Delta(E)=\mathbb{F}_q$ if $|E|\gtrsim q^{\frac{d+1}{2}}$.  In the continuous context, the Falconer conjecture asserts that any compact set in $\mathbb{R}^d$ of Hausdorff dimension $>\frac{d}{2}$ should have a distance set with positive (one dimensional) Lebesgue measure.  Because of this, it is tempting to make an analogous conjecture in the finite field case, asserting that one should be able to improve the exponent $\frac{d+1}{2}$ obtained in \cite{IR} to $\frac{d}{2}$.  In the case when $d$ is odd, however, this turns out to not be possible; an example constructed in \cite{Sharpness} by Hart, Iosevich, Koh, and Rudnev shows that for any $s<\frac{d+1}{2}$ and any constant $c$, a set $E$ can be constructed with $|E|\geq \frac{c}{2}q^{\frac{d+1}{2}}$ and $|\Delta(E)|<cq$.  Thus, even if we only ask for a positive proportion of distances, the exponent $\frac{d+1}{2}$ cannot be improved in odd dimensions.  The deviation from the continuous case is a consequence of the arithmetic of finite fields; an example is constructed using the existence of non-trivial spheres of radius zero.  \\

In even dimensions, one can still hope to improve the $\frac{d+1}{2}$ exponent.  In \cite{WolffExponent}, Chapman, Erdogan, Hart, Iosevich, and Koh study the $d=2$ case, and prove that if $q\equiv 3 \text{ mod } 4$ one has $|\Delta(E)|\gtrsim q$ whenever $|E|\gtrsim q^{4/3}$.  The same exponent was later obtained by Bennett, Hart, Iosevich, Pakianathan, and Rudnev in \cite{GroupActions} without the assumption that $q\equiv 3 \text{ mod } 4$.  Several interesting variants of the distance problem have also been studied in the finite field context.  In \cite{Vinh2}, Pham, Phuong, Sang, Valculescu, and Vinh consider distances between points and lines, and prove that if $P$ and $L$ are sets of points and lines, respectively, with $|P|\cdot |L|$ sufficiently large, then one obtains a positive proportion of distances.  In \cite{Vinh3}, Lund, Pham, and Vinh study the problem where distance is replaced by a finite field analogue of angle(defined by dividing the dot product by the lengths, in analogue to the usual geometric formula for the sine of the angle between vectors).  The authors obtain a non-trivial exponent which ensures a positive proportion of these angles are obtained.  In this paper, we will consider another variant of this problem, where we replace distances between pairs of points with congruence classes of arbitrary configurations of points.\\

Given $(k+1)$ point configurations $x$ and $y$, i.e. tuples of vectors 

\[
x=(x^1,...,x^{k+1}), y=(y^1,...,y^{k+1})\in(\mathbb{F}_q^d)^{k+1},
\]

say $x$ and $y$ are congruent and write $x\sim y$ if there is a translation $z\in\mathbb{F}_q^d$ and rotation $\theta\in O(\mathbb{F}_q^d)$ such that for each index $i$, we have $y^i=\theta x^i+z$.  Given a set $E\subset \mathbb{F}_q^d$, let $\Delta_k(E)$ be the set of congruence classes of tuples where each entry comes from $E$.  Since two pairs of points are congruent if and only if the distance between the points of each pair is the same, the set $\Delta(E)$ can be identified with $\Delta_1(E)$, and $|\Delta_1(\mathbb{F}_q^d)|=q$.  More generally, for $k\leq d$, two non-degenerate $(k+1)$-point configurations are determined by the $\binom{k+1}{2}$ distances between pairs of distinct points, so $|\Delta_k(\mathbb{F}_q^d)|\approx q^{\binom{k+1}{2}}$.  \\

This viewpoint suggests a generalization of the Erdos-Falconer problem; namely, one can ask how large $E$ must be so that $\Delta_k(E)=\Delta_k(\mathbb{F}_q^d)$, or at least $|\Delta_k(E)|\gtrsim |\Delta_k(\mathbb{F}_q^d)|$.  Hart and Iosevich studied this problem in \cite{HI} and found that if $|E|\gtrsim q^{\frac{kd}{k+1}+\frac{k}{2}}$, then $E$ contains an isometric copy of every non-degenerate $(k+1)$ point configuration, i.e., every such configuration which does not live in an $r$-dimensional affine subspace for any $r<k$.  It is easy to prove that almost all configurations are non-degenerate, i.e. the number of degenerate configurations is $o(|\Delta_k(\mathbb{F}_q^d)|)$.  However, by inspection, we see this exponent is only non-trivial if $d>\binom{k+1}{2}$.  Bennett, Hart, Iosevich, Pakianathan, and Rudnev later found in \cite{GroupActions} that one can recover a positive proportion of all congruence classes, for any $d\geq k\geq 2$, if $|E|\gtrsim q^{d-\frac{d-1}{k+1}}$.  We observe that this exponent is always non-trivial.  In \cite{Vinh1}, Duc Hiep Pham, Thang Pham, and Le Anh Vinh use methods from spectral graph theorm to study the problem when $E=A^d$ for some $A\subset\mathbb{F}^q$, and improve the exponent to $\frac{kd}{k+1-1/d}$ in this case.\\

The $d-\frac{d-1}{k+1}$ exponent in the $k\leq d$ case is proved by identifying congruence classes with distance types; more precisely, one identifies a congruence class of $(k+1)$ point configurations with the $(k+1)\times (k+1)$ matrix with $(i,j)$ entry $\|x^i-x^j\|$.  When $k>d$, this results in an overdetermined system, which means the total number of congruence classes is no longer $q^{\binom{k+1}{2}}$.  In order to extend these results to the $k>d$ case, we need to first determine the size of $|\Delta_k(\mathbb{F}_q^d)|$ for $k>d$.  Note that the space of rotations, which we denote $O(\mathbb{F}_q^d)$, has dimension $\binom{d}{2}$ and the space of translations has dimension $d$, so the space of isometries has dimension $\binom{d}{2}+d=\binom{d+1}{2}$.  So, heuristically, when we take the $d(k+1)$ dimensional space of $(k+1)$-point configurations and quotient by the $\binom{d+1}{2}$ dimensional space of isometries, we expect to get $|\Delta_k(\mathbb{F}_q^d)|\approx q^{d(k+1)-\binom{d+1}{2}}$.  In this paper, we prove that this heuristic is indeed correct, and that the $d-\frac{d-1}{d+1}$ exponent works in the $k\geq d$ case as well.  \\

In the continuous setting, Chatzikonstantinou, Iosevich, Mkrchyan, and Pakianathan \cite{Rigidity} study the problem in the continuous setting for $k\geq d$.  In that paper, the problem of the overdetermined system is resolved using the concept of combinatorial rigidity.  An adequate definition of rigidity in the continuous setting is fairly technical, but the idea is that given $(k+1)$ points, one considers a graph $G$ with those points as vertices, and an edge set $\mathcal{E}$.  They want this graph to be minimal with respect to the property that if $d:\mathcal{E}\to\mathbb{R}_{\geq 0}$ is given, the set of configurations which satisfy $\|x^i-x^j\|=d(i,j)$ for all $(i,j)\in\mathcal{E}$ contains representatives of only finitely many equivalence classes.  In other words, they prove that one can pass to tuples of distances in the $k\geq d$ case by considering not all distances, but rather the distances between a strategically chosen subset of the set of all distinct pairs.  Using this framework, they prove that the set of congruence classes of $(k+1)$ point configurations can be identified with an $m$-dimensional vector space with a natural measure, where

\[
m=d(k+1)-\binom{d+1}{2},
\]

in keeping with the aforementioned heuristic.  The authors then prove that if a compact set $E\subset\mathbb{R}^d$ has Hausdorff dimension $\geq d-\frac{1}{k+1}$ then the set of congruence classes determined by tuples from $E$, viewed as a subset of the aforementioned $m$-space, has positive measure.  \\

Our goal here is to obtain an analogous result in the finite field setting.  Many of the technical issues that arise in the continuous setting are nonexistant here, so the proof is much simpler.  In particular, we work with congruence classes directy rather than explicitly dealing with rigidity or passing to distances, and reduce the problem to counting how many configurations are related by a given rigid motion, then use Fourier analysis to get an adequate bound on our counting function.  Our first step is to prove that our heuristic about the total number of congruence classes is correct; to this end, we have the following theorem.

\begin{thm}
Let $q$ be a power of an odd prime.  For any $k\geq d$, we have $|\Delta_k(\mathbb{F}_q^d)|\approx q^{d(k+1)-\binom{d+1}{2}}$.
\end{thm}

Once we prove this, we proceed to the main theorem:

\begin{thm}
Let $k\geq d\geq 2$, let $q$ be a power of an odd prime, and suppose $E\subset \mathbb{F}_q^d$ satisfies $|E|\gtrsim q^s$, with 

\[
s=d-\frac{d-1}{k+1}.
\]

Then 

\[
|\Delta_k(E)|\gtrsim 2^{-k^2}|\Delta_k(\mathbb{F}_q^d)|.
\]

\end{thm}

Before we continue, we want to make a couple of remarks about this theorem.  First, observe that by Theorem 1, the right hand side is $2^{-k^2}q^{d(k+1)-\binom{d+1}{2}}$.  We also remark that throughout, when we make statements including notation like $\lesssim$ and $\approx$ which involve implicit constants, we allow these constants to depend on $d$.  In previous work, where $k\leq d$, this means constants are allowed to depend on $k$ as well.  However, we will always make the dependence on $k$ explicit, since clearly issues will arise if the number of points in our configruations becomes too large compared to the size of the space we are working in.  In particular, we see from the statement of Theorem 2 that if $k\gtrsim \sqrt{\log q}$ then we are losing a power of $q$ in our estimate.  Thus, we want to think of both $d$ and $k$ as being much smaller than $q$.  \\

To prove Theorem 1, we first establish that there are $\approx q^{d(k+1)-\binom{d+1}{2}}$ congruence classes of non-degenerate tuples (i.e., tuples with full affine span).  We expect this to be our main term, however it is somewhat technical to estimate the number of congruence classes of degenerate tuples.  This is done in section 2.  Theorem 2 is proved by reducing matters to counting pairs $(u,v)\in E$ such that a given rigid motion maps $u$ to $v$.  In section 3, we prove some computational lemmas about the Fourier transform of that counting function, and in section 4 we give the proof of Theorem 2 using these computations.

\section{Size of $\Delta_k(\mathbb{F}_q^d)$}

The purpose of this section is to prove Theorem 1.  We proceed by considering pinned configurations.

\begin{dfn}
Given a $(k+1)$-point configuration $x=(x^1,\cdots,x^k,x^{k+1})\in(\mathbb{F}_q^d)^{k+1}$, the corresponding pinned configuration is defined to be

\[
\widetilde{x}:=(x^1-x^{k+1},\cdots,x^k-x^{k+1},0)\in(\mathbb{F}_q^d)^{k+1}.
\]
\end{dfn}

\begin{lem}
For any $x,y\in (\mathbb{F}_q^d)^{k+1}$, we have $x\sim y$ if and only if there is a rotation $\theta$ such that $\theta\widetilde{x}=\widetilde{y}$.  In particular, for any $y$,

\[
|\{x:x\sim y\}|=q^d|\{\theta \widetilde{y}:\theta\in O(\mathbb{F}_q^d)\}|.
\]
\end{lem}

\begin{proof}
If $x\sim y$, there exist $\theta$ and $z$ such that for each $i$, $y^i=\theta x^i+z$.  Therefore,

\[
\theta(x^i-x^{k+1})=(y^i-z)-(y^{k+1}-z)=y^i-y^{k+1}.
\]

Conversely, if $\theta\widetilde{x}=\widetilde{y}$, then

\[
\theta x^i-\theta x^{k+1}=y^i-y^{k+1},
\]

or

\[
y^i=\theta x^i+(y^{k+1}-\theta x^{k+1}).
\]
\end{proof}

\begin{dfn}
For $1\leq r< d$, let $D_r\subset (\mathbb{F}_q^d)^{k+1}$ be the set of configurations which span an $r$-dimensional affine subspace of $\mathbb{F}_q^d$.  Let 

\[
D=\bigcup_{1\leq r<d}D_r.
\]

\end{dfn}

\begin{thm}
We have $|\Delta_k((\mathbb{F}_q^d)^{k+1}\setminus D)|\approx q^{d(k+1)-\binom{d+1}{2}}$.
\end{thm}

\begin{proof}
Let $x\in (\mathbb{F}_q^d)^{k+1}\setminus D$.  It follows that some $d$ of the first $k$ entries of $\widetilde{x}$ constitute a basis of $
\mathbb{F}_q^d$.  Therefore, the map $\theta\mapsto \theta x$ is injective.  It follows from the previous lemma that the congruence class containing $x$ has size $\approx q^{\binom{d+1}{2}}$.  By simple counting, $|(\mathbb{F}_q^d)^{k+1}\setminus D|\approx q^{d(k+1)}$.  The result follows.
\end{proof}

To summarize, we have decomposed $(\mathbb{F}_q^d)^{k+1}$ into a set $D$ of degenerate tuples (i.e., tuples with affine span of dimension $<d$), and the remaining non-degenerate tuples $(\mathbb{F}_q^d)^{k+1}\setminus D$ (tuples with full affine span).  We have proved that the non-degenerate part (which we expect to be the main term) determines $\approx q^{d(k+1)-\binom{d+1}{2}}$ congruence classes.  It remains to prove that the degenerate part $D$ determines $\lesssim q^{d(k+1)-\binom{d+1}{2}}$ congruence classes.  Because we allow our constants to depend on $d$, it is enough tho prove that $|\Delta_k(D_r)|\lesssim q^{d(k+1)-\binom{d+1}{2}}$ for each fixed $r<d$.  \\

Proving this requires us to consider more general bilinear forms.  Recall that throughout this paper, $q$ is a power of an odd prime, as many of the following results assume we are in a field of characteristic not equal to 2.

\begin{dfn}
Let $g:\mathbb{F}_q^d\times\mathbb{F}_q^d\to\mathbb{F}_q$ be a symmetric bilinear form.  The \textbf{kernel} of $g$, denoted $\ker g$, is the set of $x\in\mathbb{F}_q^d$ such that for all $y\in\mathbb{F}_q^d$, we have $g(x,y)=0$.  We say $g$ is \textbf{non-degenerate} if $\ker g=0$.
\end{dfn}

\begin{lem}[\cite{Witt}, Proposition 4.1]
Let $V$ be a subspace of $\mathbb{F}_q^d$ with symmetric bilinear form $g$.  There exists a subspace $W$ of $V$ such that 

\[
V=\ker g \oplus Q.
\]
\end{lem}

We also have Witt's theorem; see, for example, \cite{Lang} chapter XV, Theorem 10.2:

\begin{thm}[Witt]
Consider $\mathbb{F}_q^d$ equipped with the dot product, and let $V,W$ be subspaces.  Any isometry $V\to W$ can be extended to an isometry $\mathbb{F}_q^d\to\mathbb{F}_q^d$.
\end{thm}

\begin{cor}
Let $V$ be a subspace of $\mathbb{F}_q^d$, and let $\text{Iso}(V)$ be the group of isometries of $V$.  Then $|\text{Iso}(V)|\gtrsim q^{\binom{r+1}{2}}$.
\end{cor}

\begin{proof}
By lemma 2, write $V=O\oplus W$ where $W$ is non-degenerate and $O$ is the kernel, and let $s=\dim O$.  Then we have a family of isometries of $V$ of the form

\[
\begin{pmatrix}
A & 0 \\
0 & B
\end{pmatrix}
\]

where $A$ is an isometry of $O$ (i.e., $A\in\text{GL}_s(\mathbb{F}_q)$) and $B$ is an isometry of $W$.  By Witt's theorem, an isometry of $W$ extends to an isometry of $\mathbb{F}_q^d$, so $B\in O(\mathbb{F}_q^{r-s})$.  Therefore, 

\[
|\text{Iso}(V)|\gtrsim |\text{GL}_s(\mathbb{F}_q)|\cdot |O(\mathbb{F}_q^{r-s})|\gtrsim |O(\mathbb{F}_q^r)|\approx q^{\binom{r+1}{2}}.
\]
\end{proof}

\begin{dfn}
Consider the action of $O(\mathbb{F}_q^d)$ on the set of $r$-dimensional subspaces of $\mathbb{F}_q^d$.  Let $\mathcal{V}_r$ be a complete set of representatives of the equivalence classes of this action.
\end{dfn}

We are now ready to prove Theorem 1.

\begin{proof}
It suffices to prove that $|\Delta_k(D_r)|\lesssim q^{d(k+1)-\binom{d+1}{2}}$ for each $1\leq r<d$.  Fix such an $r$, and fix $x,y\in D_r$.  Then $\widetilde{x}$ and $\widetilde{y}$ span an $r$-dimensional subspaces, so there are isometries $T,S$ and $V, W\in\mathcal{V}_r$ such that $Tx\in V$ and $Sy\in W$.  It follows easily from Witt's theorem, and the definition of $\mathcal{V}_r$, that $x\sim y$ if and only if $T=S$, $V=W$, and there is an isometry of $V$ which maps $Tx$ to $Ty$.  This means that the number of $r$-dimensional congruence classes is

\[
|\Delta_k(D_r)|\approx \sum_{V\in\mathcal{V}_r}\frac{q^{r(k+1)}}{|\text{Iso}(V)|}\lesssim \sum_{V\in\mathcal{V}_r} q^{r(k+1)-\binom{r+1}{2}}\lesssim q^{d(k+1)-\binom{d+1}{2}}|\mathcal{V}_r|.
\]

Note that $|\mathcal{V}_r|$ depends on $d$ but not $k$ or $q$, so $|\mathcal{V}_r|\approx 1$.
\end{proof}

\section{Lemmas}

For $f:\mathbb{F}_q^d\to \mathbb{C}$, the Fourier transform of $f$ is given by

\[
\widehat{f}(m)=\frac{1}{q^d}\sum_{x\in\mathbb{F}_q^d}\chi(-x\cdot m)f(x),
\]

where $\chi$ is an additive character on $\mathbb{F}_q$.  For basic facts about the Fourier transform in this context see, for example, \cite{HI}.  Throughout this section, let $E\subset\mathbb{F}_q^d$ be fixed.  For $z\in\mathbb{F}_q^d$ and a rotation $\theta\in O(\mathbb{F}_q^d)$, let
\[
\nu_\theta(z)=|\{(u,v)\in E\times E:u-\theta v=z\}|.
\]

We first prove some lemmas about this function.  Throughout, we identify the set $E$ with it's characteristic function, and let $\widehat{E}$ be the Fourier transform of this function.  

\begin{lem}
We have
\[
\widehat{\nu_\theta}(m)=q^d\widehat{E}(m)\widehat{E}(\theta^{-1}m),
\]
and in particular
\[
\widehat{\nu_\theta}(0)=\frac{|E|^2}{q^d}.
\]
\end{lem}

\begin{proof}
Using the fact that

\[
\sum_w \chi(x\cdot w)=
\begin{cases}
0, & x\neq 0 \\
q^d, & x=0
\end{cases}
\]

we have

\begin{align*}
\nu_\theta(z)&=\sum_{\substack{u,v \\ u-\theta v=z}}E(u)E(v) \\
&=\sum_{u,v}\frac{1}{q^d}\sum_w \chi(-w\cdot u)\chi(w\cdot \theta v)\chi(w\cdot z)E(u)E(v).
\end{align*}

The Fourier transform is then

\begin{align*}
\widehat{\nu_\theta}(m)=&\frac{1}{q^d}\sum_z \chi(-m\cdot z)\nu_\theta(z) \\
=&\frac{1}{q^{2d}}\sum_{u,v}\sum_w\sum_z \chi((w-m)\cdot z)\chi(-w\cdot u)\chi(w\cdot \theta v)E(u)E(v) \\
=&q^d\left(\frac{1}{q^d}\sum_u \chi(-m\cdot u)E(u)\right)\left(\frac{1}{q^d}\overline{\sum_v \chi(-\theta^{-1}m\cdot v)E(v)}\right) \\
=&q^d\widehat{E}(m)\overline{\widehat{E}(\theta^{-1}m)}.
\end{align*}

Going from line 2 to 3, we run the sum in $z$ first, and the term $\chi((w-m)\cdot z)$ guarantees this will be zero unless $w=m$.  We also use the orthogonality of $\theta$ to get $m\cdot \theta v=\theta^{-1}m\cdot v$.

\end{proof}

\begin{lem}
\[
\tag{a} \sum_{\theta,z}\nu_\theta(z)^2\lesssim \frac{|E|^4}{q^{d-\binom{d}{2}}}+q^{\binom{d}{2}+1}|E|^2
\]
\[
\tag{b} \sum_{\theta,z}\left(\nu_\theta(z)-\frac{|E|^2}{q^d}\right)^2\lesssim q^{\binom{d}{2}+1}|E|^2
\]
\end{lem}

\begin{proof}

We have 

\begin{align*}
\sum_{\theta,z}\nu_\theta(z)^2 &=q^d\sum_{\theta,m}|\widehat{\nu_\theta}(m)|^2 \\
&=q^d\sum_{\theta}|\widehat{\nu_\theta}(0)|^2+q^d\sum_\theta \sum_{m\neq 0}|\widehat{\nu_\theta}(m)|^2 \\
&\approx \frac{|E|^4}{q^{d-\binom{d}{2}}}+q^d\sum_\theta \sum_{m\neq 0}|\widehat{\nu_\theta}(m)|^2.
\end{align*}

Using Lemma 3, the second term is

\begin{align*}
&q^d\sum_\theta\sum_{m\neq 0}|\widehat{\nu_\theta}(m)|^2 \\
=&q^d\sum_\theta\sum_{m\neq 0}|q^d\widehat{E}(m)\overline{\widehat{E}(\theta^{-1}m)}|^2 \\
=&q^{3d}\sum_{m\neq 0}|\widehat{E}(m)|^2\sum_\theta |\widehat{E}(\theta m)|^2.
\end{align*}

We can rewrite this as

\[
q^{3d}\sum_{t\in\mathbb{F}_q}\sum_{\substack{m\neq 0 \\ \|m\|=t}}|\widehat{E}(m)|^2 \sum_{\substack{l\neq 0 \\ \|l\|=t}}\sum_{\substack{\theta \\ \theta m=l}}|\widehat{E}(l)|^2
\]

Since the terms in the last sum do not depend on $\theta$, the last sum simply counts the number of $\theta$ that map a given element of $\mathbb{F}_q^d$ to one of the same length.  For $d>2$ the rotations taking a given $m$ to a given $l$ are unique up to the coset of $O(\mathbb{F}_q^{d-1})$, so there are $\binom{d-1}{2}=\binom{d}{2}-d+1$ such rotations.  Therefore, the above quantity is

\[
q^{2d+\binom{d}{2}+1}\sum_{t\in\mathbb{F}_q}\left(\sum_{\substack{m\neq 0 \\ \|m\|=t}}|\widehat{E}(m)|^2\right)^2.
\]

Dominating the sums

\[
\sum_{t\in\mathbb{F}_q}\sum_{\substack{m\neq 0 \\ \|m\|=t}}|\widehat{E}(m)|^2
\]

and

\[
\sum_{\substack{m \\ \|m\|=t}}|\widehat{E}(m)|^2
\]

by sums over the entire vector space and using Plancherel, the previous quantity is $\lesssim q^{\binom{d}{2}+1}|E|^2$.  This proves part (a).  Part (b) is proved similarly; subtracting the zero Fourier coefficient just means the $m=0$ term is absent.
\end{proof}

\begin{lem}
For all $k\in\mathbb{N}$, we have
\[
\sum_{\theta, z}\nu_\theta(z)^{k+1}\lesssim 2^{k^2}\left(q^{\binom{d}{2}+1}|E|^{k+1}+\frac{|E|^{2(k+1)}}{q^{d(k+1)-\binom{d+1}{2}}}\right).
\]
\end{lem}

\begin{proof}
We prove the lemma by induction.  The $k=1$ case is simply Lemma 4a.  Assume the statement holds for $1,...,k-1$.  We have

\begin{align*}
&\sum_{\theta, z}\nu_\theta(z)^{k+1} \\
=& \sum_{\theta, z}\left(\nu_\theta(z)-\frac{|E|^2}{q^d}\right)^{k}\nu_\theta(z)+\sum_{j=0}^{k-1}(-1)^{k-j+1}\binom{k}{j}\frac{|E|^{2(k-j)}}{q^{d(k-j)}}\sum_{z,\theta}\nu_\theta(z)^{j+1} \\
=& A+B.
\end{align*}

To estimate $A$, we use Lemma 4b:

\begin{align*}
A=&\sum_{\theta, z}\left(\nu_\theta(z)-\frac{|E|^2}{q^d}\right)^{k}\nu_\theta(z) \\
\lesssim & 2^{k-1}|E|^{k-1}\sum_{\theta, z}\left(\nu_\theta(z)-\frac{|E|^2}{q^d}\right)^2 \\
\lesssim & 2^{k-1}q^{\binom{d}{2}+1}|E|^{k+1}.
\end{align*}

To estimate $B$, using the inductive hypothesis we have

\begin{align*}
B=& \sum_{j=0}^{k-1}(-1)^{k-j+1}\binom{k}{j}\frac{|E|^{2(k-j)}}{q^{2(k-j)}}\sum_{z,\theta}\nu_\theta(z)^{j+1} \\
\lesssim& \sum_{j=0}^{k-1}(-1)^{k-j+1}\binom{k}{j}\frac{|E|^{2(k-j)}}{q^{2(k-j)}}2^{j^2}\left(q^{\binom{d}{2}+1}|E|^{j+1}+\frac{|E|^{2(j+1)}}{q^{d(j+1)-\binom{d+1}{2}}}\right)
\end{align*}

It is easy to show that if $j<k/2$, then $\binom{k}{j}\leq \binom{k}{j+1}$.  So, the binomial coefficients are maximized when $j$ is $k/2$ for $k$ even, or the floor or ceiling of $k/2$ if $k$ is odd.  Using this and Stirling's formula gives $\binom{k}{j}\lesssim 2^k$.  This gives

\begin{align*}
&\left|\sum_{j=0}^{k-1}(-1)^{k-j+1}\binom{k}{j}\frac{|E|^{2(k-j)}}{q^{d(k-j)}}2^{j^2}\left(q^{\binom{d}{2}+1}|E|^{j+1}+\frac{|E|^{2(j+1)}}{q^{d(j+1)-\binom{d+1}{2}}}\right)\right| \\
\lesssim & 2^{k+(k-1)^2}\sum_{j=0}^{k-1}\frac{|E|^{2(k-j)}}{q^{d(k-j)}}\left(q^{\binom{d}{2}+1}|E|^{j+1}+\frac{|E|^{2(j+1)}}{q^{d(j+1)-\binom{d+1}{2}}}\right) \\
\lesssim &k2^{k+(k-1)^2}\max_{0\leq j<k}\left(\frac{|E|^{2k-j+1}}{q^{d(k-j)-\binom{d}{2}-1}}+\frac{|E|^{2(k+1)}}{q^{d(k+1)-\binom{d+1}{2}}}\right)
\end{align*}

It is easily proved by induction that $\log_2(k)+1\leq k$ for $k\in\mathbb{N}$.  Therefore, $k2^{k+(k-1)^2}\leq 2^{k^2}$.  The second term in the parentheses does not depend on $j$.  For the first term, by direct computation we have

\[
\frac{|E|^{2k-j+1}}{q^{d(k-j)-\binom{d}{2}-1}}\lesssim q^{\binom{d}{2}+1}|E|^{k+1},
\]

so the claim is proved.
\end{proof}

\section{Proof of main theorem}

We are now ready to prove Theorem 2.

\begin{proof}

Our goal is to obtain a lower bound for $|\Delta_k(E)|$ in terms of $|E|,q,k$ which will be $\gtrsim 2^{-k^2}q^{d(k+1)-\binom{d+1}{2}}$ when $|E|>q^s$, with $s=d-\frac{d-1}{k+1}$.  We first apply Cauchy-Schwarz.  This gives

\begin{align*}
|E|^{2(k+1)}=&\left(\sum_{\mathcal{C}\in\Delta_k(E)} |\mathcal{C}|\right)^2 \\
\leq & |\Delta_k(E)|\cdot\sum_{\mathcal{C}\in\Delta_k(E)}|\mathcal{C}|^2 \\
=& |\Delta_k(E)|\cdot\sum_{\mathcal{C}\in \Delta_k(E)}\sum_{x,y\in \mathcal{C}} 1 \\
=& |\Delta_k(E)|\cdot\sum_{x,y\in E^{k+1}}\sum_{\substack{\mathcal{C}\in\Delta_k(E) \\ x,y\in \mathcal{C}}} 1 \\
=& |\Delta_k(E)|\cdot\sum_{\substack{x,y\in E^{k+1} \\ x\sim y}} 1.
\end{align*}

Next, we observe that the relation $x\sim y$ is equivalent to the existance of a rotation $\theta$ and a translation $z\in\mathbb{F}_q^d$ such that $x=\theta y+z$.  Moreover, $x=\theta y+z$ means for each $i$, $x^i=\theta y^i+z^i$.  Therefore,

\begin{align*}
\{(x,y)\in E^{k+1}\times E^{k+1}:x\sim y\}=&\{(x,y)\in E^{k+1}\times E^{k+1}:\exists\ \theta,z\ x-\theta y=z\} \\
=&\bigcup_{\theta, z}\{(x,y)\in E^{k+1}\times E^{k+1}:x-\theta y=z\} \\
=&\bigcup_{\theta, z}\{(u,v)\in E\times E:u-\theta v=z\}^{k+1}
\end{align*}

Putting this together and using Lemma 5, we have

\[
|E|^{2(k+1)}\leq |\Delta_k(E)|\sum_{\theta, z}\nu_\theta(z)^{k+1}\lesssim |\Delta_k(E)|2^{k^2}\left(q^{\binom{d}{2}+1}|E|^{k+1}+\frac{|E|^{2(k+1)}}{q^{d(k+1)-\binom{d+1}{2}}}\right).
\]

If the second term on the right is larger, we get $|\Delta_k(E)|\gtrsim 2^{-k^2}q^{d(k+1)-\binom{d+1}{2}}$ for free.  If the first term is larger, we have

\[
|\Delta_k(E)|\gtrsim 2^{-k^2}\frac{|E|^{k+1}}{q^{\binom{d}{2}+1}}.
\]

If $|E|\gtrsim q^s$, then this gives

\[
|\Delta_k(E)|\gtrsim 2^{-k^2}q^{s(k+1)-\binom{d}{2}-1}.
\]

This will be $\gtrsim 2^{-k^2}q^{d(k+1)-\binom{d+1}{2}}$ if $s=d-\frac{d-1}{k+1}$.
\end{proof}


\begin{thebibliography}{9}

\bibitem{GroupActions}
  M. Bennett, D. Hart, A. Iosevich, J. Pakianathan, M. Rudnev,
  \textit{Group actions and geometric combinatorics in $\mathbb{F}_q^d$}, Forum Math., 29(1):91-110, 2017.
	
\bibitem{WolffExponent}
  Jeremy Chapman, M. Barak Erdogan, Derrick Hart, Alex Iosevich, Doowon Koh,
  \textit{Pinned distance sets, $k$-simplices, Wolff's exponent in finite fields and sum product estimates}, Mathematische Zeitschrift, Math. Z. 271 (2012), no. 1-2, 63-93

\bibitem{Rigidity}
  Nikolaos Chatzikonstantinou, Alex Iosevich, Sevak Mkrtchyan, Jonathan Pakianathan,
  \textit{Rigidity, Graphs and Hausdorff dimension}, https://arxiv.org/pdf/1708.05919.pdf
	2018
	
\bibitem{Witt}
  Pete L. Clark,
  \textit{Quadratic forms chapter I: Witt's theory}, notes http://math.uga.edu/~pete/quadraticforms.pdf
	
\bibitem{HI}
  Derrick Hart, Alex Iosevich,
  \textit{Ubiquity of simplices in subsets of vector spaces over finite fields}, Anal. Math. 34 (2008), no. 1, 29–38
	
\bibitem{Sharpness}
  Derrick Hart, Alex Iosevich, Doowon Koh, Misha Rudnev,
  \textit{Averages over hyperplanes, sum product theory in vector spaces over finite fields, and the Erdos-Falconer distance conjecture}, Transactions of the AMS, \textbf{363} (2011) 3255-3275
	
\bibitem{IR}
  Alex Iosevich, Misha Rudnev
  \textit{Erdos distance problem in vector spaces over finite fields}, Trans. Amer. Math. Soc. \textbf{359} (2007), no. 12, 6127-6142
	2005
	
\bibitem{Lang}
  S. Lang, \textit{Algebra, revised 3rd edition}, Springer Science + Business Media LLC, 2002
	
\bibitem{Vinh3}
	Ben Lund, Thang Pham, Le Anh Vinh, \textit{Distinct spreads in vector spaces over finite fields}, Discrete Appl. Math. 239 (2018), 154–158.
	
\bibitem{Vinh1}
	H. Pham, T. Pham, L.A. Vinh, \textit{An improvement on the number of simplices in $\mathbb{F}_q^d$}, Discrete Appl. Math. 221 (2017) 95–105. MR3612592
	
\bibitem{Vinh2}
	Thang Pham, Nguyen Duy Phuong, Nguyen Minh Sang, Claudia Valculescu, Le Anh Vinh, \textit{Distinct distances between points and lines in $\mathbb{F}_q^2$}, Forum Math. 30 (2018), no. 4, 799–808


\end{thebibliography}
\end{document}